\documentclass[12pt]{article}

\usepackage{amsfonts}
\usepackage{amsmath}
\usepackage{amssymb}
\usepackage{color}
\usepackage{amsthm}
\usepackage{hyperref}
\usepackage{makecell}
\usepackage{graphicx}
\usepackage{multirow}

\theoremstyle{plain}
\newtheorem{theorem}{Theorem}

\newtheorem{lemma}[theorem]{Lemma}

\theoremstyle{definition}

\theoremstyle{remark}
\newtheorem{remark}[theorem]{Remark}

\title{The binomial distribution maximises entropy among sums of negatively related indicators with fixed mean and support} 

\author{Fraser Daly\footnote{Department of Actuarial Mathematics and Statistics, and the Maxwell Institute for Mathematical Sciences, Heriot--Watt University, Edinburgh EH14 4AS, UK. E-mail: F.Daly@hw.ac.uk}} 

\date{\today}

\begin{document}

\maketitle

\noindent{\bf Abstract} 
We show that the binomial distribution supported on $\{0,1,\ldots,n\}$ and with mean $np$ for some $0<p<1$ maximises entropy within the class of sums of negatively related indicator random variables with this mean and support. Our results generalise the main result of Yu (2008), who showed that this binomial distribution maximises entropy within the smaller class of ultra-log-concave random variables of order $n$. Our proof makes use of a Markov chain construction of Yu (2008), together with a stochastic ordering argument. 

\vspace{12pt}


\vspace{12pt}


\section{Introduction and main results}

Let $W$ be a random variable supported on $\{0,1,\ldots,n\}$ for some $n\geq1$ and with mean $\mathbb{E}W=\mu\in(0,n)$. We write $f_k=\mathbb{P}(W=k)$, $k=0,1,\ldots,n$, for the mass function of $W$. We will show that if $W$ satisfies a certain stochastic ordering condition, then the entropy of $W$, denoted by
\[
H(W)=-\sum_{k=0}^nf_k\log f_k\,,
\]
is bounded above by the entropy of a binomial distribution whose support and mean match those of $W$. That is, the binomial distribution maximises entropy within the class of random variables $W$ supported on $\{0,1,\ldots,n\}$ with a given mean and satisfying our stochastic ordering condition. As we will see in Remark \ref{rem:ULC} below, this generalises previous maximum entropy results for the binomial due to \cite{h01} and \cite{y08}, since this class includes sums of independent indicator random variables and ultra-log-concave random variables of order $n$.

We will show in particular that sums of negatively related indicator random satisfy our stochastic ordering condition. That is, we will focus on the case where $W=X_1+\cdots+X_n$ for some random variables $X_1,\ldots,X_n$ each supported on \{0,1\} and which satisfy the property of being negatively related:
\begin{equation}\label{eq:NegRel}
\mathbb{E}[g(X_1,\ldots,X_{i-1},X_{i+1},\ldots,X_n)|X_i=1]\leq\mathbb{E}[g(X_1,\ldots,X_{i-1},X_{i+1},\ldots,X_n)]\,,
\end{equation}
for each $i=1,\ldots,n$ and all increasing functions $g:\{0,1\}^{n-1}\to\mathbb{R}$. See Section 2.4 of \cite{e05} for a discussion of this property.

There are many well-known examples of negatively related indicator random variables, including the following:
\begin{itemize}
\item If $X_1,\ldots,X_n$ are independent, then \eqref{eq:NegRel} holds with equality. In particular, a binomial random variable can be written as a sum of negatively related indicator random variables.
\item The random variables $X_1,\ldots,X_n$ are said to be negatively associated if
\[
\mathbb{E}[f(X_i:i\in A)g(X_i:i\in B)]\leq\mathbb{E}[f(X_i:i\in A)]\mathbb{E}[g(X_i:i\in B)]
\]
for all bounded and increasing functions $f$ and $g$, and disjoint sets $A,B\subseteq\{1,\ldots,n\}$. This property was introduced and studied by Joag-Dev and Proschan \cite{jp83}, and is known to hold in a variety of settings. For example, it arises from the multinomial distribution and from the joint distribution of ranks in a random sample. The negative association property is preserved under taking increasing functions of disjoint subsets of random variables. It is known that negatively associated indicator random variables are negatively related: see Theorem 2.20 of \cite{e05}. 
\item Letting $(\sigma_1,\ldots,\sigma_n)$ be a uniformly chosen permutation of $\{1,\ldots,n\}$, define $X_i$ to be the indicator that $\sigma_i$ is at most $a_i$ for some given numbers $a_1,\ldots,a_n$. Then $X_1,\ldots,X_n$ are negatively related; see Section 4.1 of \cite{bhj92}.
\item If $m$ balls are each randomly placed into one of $n$ urns, let $X_i$ be an indicator that urn $i$ contains at least $c$ balls. These $X_i$ are negatively related (see Section 4 of \cite{pp02}) and their sum counts the number of urns which contain at least $c$ balls.
\item Given an urn initially containing $N$ balls of $n$ different colours, on each of $m$ draws a ball is selected uniformly at random and returned to the urn along with an additional ball of the same colour. Letting $X_i$ be an indicator that no ball of colour $i$ was drawn during this procedure, these $X_i$ are known to be negatively related: see Section 6.3 of \cite{bhj92}. Their sum represents the number of different colours that were never drawn.
\item An $r\times n$ matrix is constructed by placing $s_k$ 1s in row $k$, with their positions chosen by sampling without replacement. For a given $m$, let $X_i$ be the indicator that column $i$ contains at most $m$ non-zero entries. These $X_i$ are negatively related (see Section 6.4 of \cite{bhj92}) and their sum is the count of such columns. 
\item Distribute $n$ points uniformly at random on the unit circle, and let $X_i$ be the indicator that the spacing between the $i$th pair of adjacent points is smaller than some given threshold. These $X_i$ are negatively related (see Section 7.1 of \cite{bhj92}) and their sum counts the number of small spacings on the circle.
\end{itemize}

In what follows we write that $Z\sim\text{Bin}(n,p)$ has a binomial distribution if $\mathbb{P}(Z=k)=\binom{n}{k}p^k(1-p)^{n-k}$ for $k=0,\ldots,n$ and some $0\leq p\leq1$. Our main result is the following:
\begin{theorem}\label{thm:NegRel}
Let $X_1,\ldots,X_n$ be negatively related indicator random variables, and $W=X_1+\cdots+X_n$ with mean $\mathbb{E}W=\mu$. Then $H(W)\leq H(Z)$, where $Z\sim\text{Bin}(n,\mu/n)$.
\end{theorem}
This will follow from the more general Theorem \ref{thm:StochOrder} below. The proofs of both of these results are given in Section \ref{sec:Proofs}, following which we give some concluding remarks and extensions in Section \ref{sec:Conc}. The proof of Theorem \ref{thm:StochOrder} is based on the hypergeometric thinning construction of Yu \cite{y08} combined with a stochastic ordering argument in the spirit of \cite{d16}; we then use an explicit construction of the relevant random variables to deduce Theorem \ref{thm:NegRel}. 

In order to state Theorem \ref{thm:StochOrder} we require some additional notation. Given a random variable $W$ supported on $\{0,1,\ldots,n\}$ and with mean $\mu$, we define the random variables $W^\star$ and $W^+$ as follows:
\begin{equation}\label{eq:Transforms}
\mathbb{P}(W^\star=k)=\frac{k\mathbb{P}(W=k)}{\mu}\text{ and }
\mathbb{P}(W^+=k)=\frac{(n+1-k)\mathbb{P}(W+1=k)}{n-\mu}\,,
\end{equation}
each for $k=1,\ldots,n$. Note that we could equivalently define $W^+$ by writing $W^+=n+1-(n-W)^\star$. The random variable $W^\star$ is the size-biased version of $W$, which has previously been used to establish maximum entropy results for the Poisson distribution via a stochastic comparison with $W+1$ \cite{d16,j07}. 

It is easily checked that if $W$ is binomial, the distributions of $W^\star$ and $W^+$ coincide. Here we will use a stochastic comparison between $W^\star$ and $W^+$ to establish our maximum entropy results for the binomial distribution. A random variable $X$ is stochastically smaller than a random variable $Y$, written $X\leq_\text{st}Y$, if $\mathbb{E}g(X)\leq\mathbb{E}g(Y)$ for all increasing functions $g$.

\begin{theorem}\label{thm:StochOrder}
Let $W$ be a random variable supported on $\{0,1,\ldots,n\}$, with mean $\mathbb{E}W=\mu$ and such that $W^\star\leq_\text{st}W^+$. Then $H(W)\leq H(Z)$, where $Z\sim\text{Bin}(n,\mu/n)$. 
\end{theorem}

\begin{remark}\label{rem:ULC}
Theorem \ref{thm:StochOrder} generalises the main result of Yu \cite{y08}, who showed that the binomial distribution maximises entropy within the class of random variables which are ultra-log-concave of order $n$. A random variable $W$ supported on $\{0,1,\ldots,n\}$ is said to be ultra-log-concave of order $n$ if
\[
\frac{\mathbb{P}(W=k+1)^2}{\binom{n}{k+1}^2}\geq\frac{\mathbb{P}(W=k)\mathbb{P}(W=k+2)}{\binom{n}{k}\binom{n}{k+2}}\,,
\]
for $k=0,\ldots,n-2$. Using the definitions \eqref{eq:Transforms}, we see that this is equivalent to
\[
\frac{\mathbb{P}(W^+=k+2)}{\mathbb{P}(W^\star=k+2)}\geq\frac{\mathbb{P}(W^+=k+1)}{\mathbb{P}(W^\star=k+1)}\,,
\]
that is, that $W^\star$ is smaller than $W^+$ in the likelihood ratio order. Recall that a random variable $X$ is smaller than $Y$ in the likelihood ratio order if $\frac{\mathbb{P}(Y=k)}{\mathbb{P}(X=k)}$ is increasing over the union of the supports of $X$ and $Y$. This likelihood ratio ordering between $W^\star$ and $W^+$ assumed by \cite{y08} is stronger than the corresponding stochastic order we assume in our Theorem \ref{thm:StochOrder}: see Theorem 1.C.1 of \cite{ss07}.
Note also that the main result of \cite{y08} is itself a generalisation of the maximum entropy result in Theorem 7 of \cite{h01}, which showed that the binomial distribution maximises entropy among sums of independent indicator random variables with fixed support and mean.

Furthermore, Theorem \ref{thm:NegRel} provides a more natural and sharper comparison than Corollary 2.3 of \cite{d16}, which shows that the entropy of a sum of negatively related indicator random variables is bounded by that of a Poisson distribution with matching mean. That a binomial random variable can itself be written as a sum of negatively related indicators makes our Theorem \ref{thm:NegRel} a more natural comparison, as well as a sharp upper bound.
\hfill $\square$
\end{remark}

\section{Proofs of main results}\label{sec:Proofs}

We will first prove Theorem \ref{thm:StochOrder}, before using that result to deduce Theorem \ref{thm:NegRel}.

\subsection{Proof of Theorem \ref{thm:StochOrder}}

We begin by recalling a Markov chain construction of Yu \cite{y08}, which we will also employ in our proof. Yu shows that when its initial distribution is ultra-log-concave of order $n$, this property is preserved by this Markov chain. We will show that an analogous property holds if the initial distribution satisfies the weaker stochastic ordering condition of Theorem \ref{thm:StochOrder}, and use this to deduce a comparison of the entropy of this initial distribution with the binomial limit of the Markov chain. With $W_0=W$ as in Theorem \ref{thm:StochOrder} and $p=\mu/n$, we define $\{W_t:t=0,1,\ldots\}$ by letting 
\[
W_{t+1}=\mathcal{H}_{n+1}(W_t+\eta_t)
\]
for $t=0,1,\ldots$, where $\eta_0,\eta_1,\ldots$ is a sequence of independent indicator random variables, each with mean $p$ and independent of all else, and $\mathcal{H}_n$ is a hypergeometric thinning operator defined for a random variable $X$ supported on $\{0,1,\ldots,n\}$ by
\begin{equation}\label{eq:Hdef}
\mathbb{P}(\mathcal{H}_nX=k)=\frac{n-k}{n}\mathbb{P}(X=k)+\frac{k+1}{n}\mathbb{P}(X=k+1)
\end{equation}
for $k=0,1,\ldots,n-1$. That is, 
\begin{equation}\label{eq:CondHyp}
\mathcal{H}_nX|X\sim\text{Hyp}(n-1,n,X)\,,
\end{equation}
where we write that $T\sim\text{Hyp}(m,N,r)$ to denote that $T$ has a hypergeometric distribution with mass function $\mathbb{P}(T=k)=\frac{\binom{r}{k}\binom{N-r}{m-k}}{\binom{N}{m}}$ and mean $\mathbb{E}T=\frac{mr}{N}$.

Let $W_0$ have support $\{0,1,\ldots,n\}$ and mean $\mu$, as usual. Then the same is true of each of the subsequent $W_t$. Further, if $W_0$ is binomial then so are each of the $W_t$ and, regardless of the initial distribution, the $W_t$ converge in distribution as $t\to\infty$ to a binomial random variable with this mean and support. See \cite{y08} for a discussion of these and further properties of this Markov chain.

We now show that if $W_0^\star\leq_\text{st}W_0^+$, then this ordering continues to hold for subsequent steps of this Markov chain. This is done by combining the following two lemmas. These play a role analogous to that of Lemma 1 of \cite{y08}, which shows that ultra-log-concavity is preserved under these same operations.
\begin{lemma}\label{lem:Pres1}
Let $X$ be a random variable supported on $\{0,1,\ldots,n\}$ and with mean $\mu$. If $X^\star\leq_\text{st}X^+$ then $(X+\eta)^\star\leq_\text{st}(X+\eta)^+$, where $\eta$ is an indicator random variable, independent of $X$, with $\mathbb{E}\eta=p=\mu/n$.
\end{lemma}
\begin{proof}
Let $K$ be an indicator random variable, independent of all else, with mean $\frac{n}{n+1}$. It is well known that we may construct $(X+\eta)^\star$ as $K(X^\star+\eta)+(1-K)(X+1)$, since we may size-bias a sum of independent random variables by choosing a summand with probability proportional to its mean and replacing it with its size-biased version (see Section 2.4 of \cite{agk19}), and since $\eta^\star=1$ almost surely. 

We show that we may similarly write $(X+\eta)^+$ as $K(X^++\eta)+(1-K)(X+1)$, by directly using the definition \eqref{eq:Transforms}: for $k=1,\ldots,n+1$ we have 
\begin{align*}
\mathbb{P}((X+\eta)^+=k)&=\frac{(n+2-k)\mathbb{P}(X+\eta+1=k)}{n+1-\mu-p}\\
&=\frac{n+2-k}{n+1-\mu-p}\left(p\mathbb{P}(X+2=k)+(1-p)\mathbb{P}(X+1=k)\right)\\
&=\frac{np}{n+1}\mathbb{P}(X^+=k-1)+\frac{n(1-p)}{n+1}\mathbb{P}(X^+=k)+\frac{1}{n+1}\mathbb{P}(X+1=k)\\
&=\frac{n}{n+1}\mathbb{P}(X^++\eta=k)+\frac{1}{n+1}\mathbb{P}(X+1=k)\,,
\end{align*}
where we recall that $p=\mu/n$ and use the independence of $X$ and $\eta$. The representation of $(X+\eta)^\star$ above may be alternatively derived directly by a similar argument.

To complete the proof, it remains only to note that since $X^\star\leq_\text{st}X^+$ we have $K(X^\star+\eta)+(1-K)(X+1)\leq_\text{st}K(X^++\eta)+(1-K)(X+1)$ using standard closure properties of the usual stochastic ordering: see Theorem 1.A.3 of \cite{ss07}.
\end{proof}
\begin{lemma}\label{lem:Pres2}
Let $X$ be a random variable supported on $\{0,1,\ldots,n\}$ and with mean $\mu$. If $X^\star\leq_\text{st}X^+$ then $(\mathcal{H}_nX)^\star\leq_\text{st}(\mathcal{H}_nX)^+$, where $\mathcal{H}_n$ is defined by \eqref{eq:Hdef}.
\end{lemma}
\begin{proof}
If $T\sim\text{Hyp}(m,N,r)$ then using the definition \eqref{eq:Transforms} it is easily checked that $T^\star\sim\text{Hyp}(m-1,N-1,r-1)+1$. It follows from mixing properties of the size-bias transformation (see Section 2.2.6 of \cite{agk19}) and the mixed hypergeometric representation \eqref{eq:CondHyp} of $\mathcal{H}_nX$ that $(\mathcal{H}_nX)^\star|X^\star\sim\text{Hyp}(n-2,n-1,X^\star-1)+1$. This can also be shown directly using an argument analogous to the construction of $(\mathcal{H}_nX)^+$ we give below.

Writing $Y=\mathcal{H}_nX$ and noting that $Y$ is supported on $\{0,1,\ldots,n-1\}$, from the definition \eqref{eq:Transforms} we have that for any function $g$,
\begin{align*}
\mathbb{E}g(Y^+)&=\frac{\mathbb{E}[(n-1-Y)g(Y+1)]}{n-1-\mathbb{E}Y}\\
&=\sum_x\frac{\mathbb{E}[(n-1-Y)g(Y+1)|X=x-1]}{n-1-\mathbb{E}Y}\mathbb{P}(X=x-1)\\
&=\sum_x\frac{\mathbb{E}[(n-1-Y)g(Y+1)|X=x-1]}{\mathbb{E}[n-1-Y|X=x-1]}\frac{\mathbb{E}[n-1-Y|X=x-1]\mathbb{P}(X=x-1)}{n-1-\mathbb{E}Y}\,.
\end{align*}
From the mixed hypergeometric representation \eqref{eq:CondHyp} of $Y=\mathcal{H}_nX$ we have that
\begin{equation}\label{eq:lem5_2}
\mathbb{E}[n-1-Y|X=x-1]=n-1-\frac{(n-1)(x-1)}{n}=\frac{n-1}{n}(n+1-x)\,,
\end{equation}
and so since
\begin{equation}\label{eq:lem5_1}
\frac{\mathbb{E}[n-1-Y|X=x-1]\mathbb{P}(X=x-1)}{n-1-\mathbb{E}Y}\propto(n+1-x)\mathbb{P}(X+1=x)
\end{equation}
and is a probability mass function, we conclude that the left-hand side of \eqref{eq:lem5_1} is equal to $\mathbb{P}(X^+=x)$. We also have that
\[
\mathbb{E}[(n-1-Y)g(Y+1)|X=x-1]=\sum_{k=x-2}^{x-1}\frac{n-1-k}{n}g(k+1)\binom{x-1}{k}\binom{n+1-x}{n-1-k}\,,
\]
so that using \eqref{eq:lem5_2} it is easily shown that
\[
\frac{\mathbb{E}[(n-1-Y)g(Y+1)|X=x-1]}{\mathbb{E}[n-1-Y|X=x-1]}=\sum_{k=x-2}^{x-1}g(k+1)\frac{\binom{x-1}{k}\binom{n-x}{n-k-2}}{\binom{n-1}{n-2}}=\mathbb{E}g(H_x+1)\,,
\]
where $H_x\sim\text{Hyp}(n-2,n-1,x-1)$. Hence, we have that $(\mathcal{H}_nX)^+|X^+\sim\text{Hyp}(n-2,n-1,X^+-1)+1$. 

To complete the proof we note that $T\sim\text{Hyp}(m,N,r)$ is stochastically increasing in its third parameter. This can be easily seen by constructing $T$ by selecting $m$ positions without replacement out of $N$ positions in total, and counting the number of the first $r\leq m$ positions which were selected. Hence, if $X^\star\leq_\text{st}X^+$ then the mixed hypergeometric distributions constructed above are such that $\text{Hyp}(n-2,n-1,X^\star-1)$ is stochastically smaller than $\text{Hyp}(n-2,n-1,X^+-1)$, so that $(\mathcal{H}_nX)^\star\leq_\text{st}(\mathcal{H}_nX)^+$, as required. 
\end{proof}
Having shown that $W_t^\star\leq_\text{st}W_t^+$ for all $t$ when $W_0=W$ satisfies the conditions of Theorem \ref{thm:StochOrder}, we argue as in Theorem 4.2 of \cite{d16} to obtain a convex ordering between $W$ and $Z\sim\text{Bin}(n,\mu/n)$. We say that a random variable $X$ is smaller than $Y$ in the convex order, written $X\leq_\text{cx}Y$, if $\mathbb{E}g(X)\leq\mathbb{E}g(Y)$ for all convex functions $g$ for which these expectations exist.  

We will make use of a further representation of the transition probabilities of our Markov chain given by Lemma 4.1 of \cite{d16}, which follows in a straightforward way from their definition:
\begin{equation}\label{eq:TransProb}
\mathbb{P}(W_t=k)-\mathbb{P}(W_{t+1}=k)=\frac{np(1-p)}{n+1}\Delta\left[\mathbb{P}(W_t^+=k)-\mathbb{P}(W_t^\star=k)\right]\,,
\end{equation}
where $\Delta g(k)=g(k+1)-g(k)$ is the forward difference operator, with inverse $\Delta^{-1}g(k)=-\sum_{x=k}^\infty g(x)$.  
\begin{lemma}\label{lem:Convex}
Let $W$ and $Z$ be as in Theorem \ref{thm:StochOrder}, and $\{W_t:t=0,1,\ldots\}$ as above. Then $W_t\leq_\text{cx}W_u$ for $t\leq u$. In particular, $W\leq_\text{cx}Z$. 
\end{lemma}
\begin{proof}
The argument follows Theorem 4.2 of \cite{d16}; we sketch it here for completeness. To establish the first statement, we show that $W_t\leq_\text{cx}W_{t+1}$ for all $t=0,1,\ldots$. Using Theorem 3.A.1 of \cite{ss07}, we do this by showing that $\sum_{x=k}^\infty\mathbb{P}(W_{t+1}\geq x)\geq\sum_{x=k}^\infty\mathbb{P}(W_{t}\geq x)$ for all $k=0,1,\ldots$. To that end, we note that Lemmas \ref{lem:Pres1} and \ref{lem:Pres2} combine to show that $W_t^\star\leq_\text{st}W_t^+$ for each $t$, or equivalently, that $\mathbb{P}(W_t^+\geq k)\geq\mathbb{P}(W_t^\star\geq k)$ for all $k$. Applying the summation operator $\Delta^{-1}$ to \eqref{eq:TransProb} twice, we have that
\[
\sum_{x=k}^\infty\left(\mathbb{P}(W_{t+1}\geq x)-\mathbb{P}(W_t\geq x)\right)=\frac{np(1-p)}{n+1}\left(\mathbb{P}(W_t^+\geq k)-\mathbb{P}(W_t^\star\geq k)\right)\geq0\,,
\]
as required. The second statement of the lemma follows from the first by taking $t=0$ and $u\to\infty$. 
\end{proof}
Now, having established that $W\leq_\text{cx}Z$, we note that the binomial distribution is log-concave, that is, satisfies $b_{k+1}^2\geq b_kb_{k+2}$ for $k=0,\ldots,n-2$, where we write $b_k=\mathbb{P}(Z=k)$. Theorem \ref{thm:StochOrder} then follows from the observation made in Lemma 1 of \cite{y09} that convex ordering implies a corresponding entropy bound when the upper-bounding random variable is log-concave. We again note the concise proof for completeness:
\begin{equation}\label{eq:LogConcave}
H(Z)=-\sum_{k=0}^nb_k\log b_k\geq-\sum_{k=0}^nf_k\log b_k\geq-\sum_{k=0}^nf_k\log f_k=H(W)\,,
\end{equation}
where the first inequality follows from the convex ordering of Lemma \ref{lem:Convex} and log-concavity of the binomial distribution, and the second from the fact that the relative entropy $\sum_{k=0}^nf_k\log(f_k/b_k)\geq0$. This completes the proof of Theorem \ref{thm:StochOrder}.

\subsection{Proof of Theorem \ref{thm:NegRel}}

We now deduce Theorem \ref{thm:NegRel} from Theorem \ref{thm:StochOrder}, beginning with an explicit construction of $W^\star$ and $W^+$ in the case where $W$ is a sum of (dependent) indicator random variables.
\begin{lemma}\label{lem:Coupling}
Let $X_1,\ldots,X_n$ be indicator random variables with $\mathbb{E}X_i=p_i$ for each $i$, and $W=X_1+\cdots+X_n$ with mean $\mathbb{E}W=\mu=p_1+\cdots+p_n$. Choose the pair $(I,J)$ independent of all else with 
\[
\mathbb{P}(I=i)=\frac{p_i}{\mu}\text{ and }\mathbb{P}(J=j)=\frac{1-p_j}{n-\mu}\,,
\]
for $i,j=1,\ldots,n$. Then we may write
\[
W^\star=1+\sum_{i\not=I}X_i^{(I,1)}\text{ and }W^+=1+\sum_{j\not=J}X_j^{(J,0)}\,,
\]
where $X_i^{(I,1)}$ is equal in distribution to $X_i|(X_I=1)$ and $X_j^{(J,0)}$ is equal in distribution to $X_j|(X_J=0)$.
\end{lemma}
\begin{proof}
The construction of the size-biased version $W^\star$ is well known: see, for example, Section 2.4 of \cite{agk19}. It may also be derived by making appropriate changes to the argument we use for $W^+$ here. 

Taking $W^+$ as in the statement of the lemma, we have that, for any $g:\{1,\ldots,n\}\to\mathbb{R}$,
\begin{align*}
\mathbb{E}g(W^+)&=\sum_{j=1}^n\mathbb{P}(J=j)\mathbb{E}[g(W+1-X_j)|X_j=0]\\
&=\frac{1}{n-\mu}\sum_{j=1}^n(1-p_j)\sum_{k=1}^{n}g(k)\mathbb{P}(W+1-X_j=k|X_j=0)\\
&=\frac{1}{n-\mu}\sum_{j=1}^n\mathbb{E}[(1-X_j)g(W+1)]\\
&=\frac{\mathbb{E}[(n-W)g(W+1)]}{n-\mu}=\sum_{k=1}^ng(k)\frac{(n+1-k)\mathbb{P}(W+1=k)}{n-\mu}\,,
\end{align*}
as required to match the definition in \eqref{eq:Transforms}.
\end{proof}
We remark that in Lemma \ref{lem:Coupling} we specify the marginal distributions of $I$ and $J$ only, and they may be constructed with whatever joint distribution (consistent with these marginals) is most convenient. For example, in the case where the $X_i$ are identically distributed (i.e., the $p_i$ are all equal) both $I$ and $J$ are uniformly distributed over $\{1,\ldots,n\}$ and we may choose $I=J$ almost surely.

To establish Theorem \ref{thm:NegRel} it remains only to show that for $W$ as in the statement of that result, $W^\star\leq_\text{st}W^+$, and then apply Theorem \ref{thm:StochOrder}. Using the coupling from Lemma \ref{lem:Coupling} to construct $W^\star=1+(W-X_I|X_I=1)$ and $W^+=1+(W-X_J|X_J=0)$, we need to show that for all increasing functions $g:\{0,1,\ldots,n-1\}\to\mathbb{R}$ we have
\[
\mathbb{E}[g(W-X_I)|X_I=1]\leq\mathbb{E}[g(W-X_J)|X_J=0]\,,
\]
where $I$ and $J$ are as in Lemma \ref{lem:Coupling}, and $X_1,\ldots,X_n$ are negatively related indicator random variables. We will need the following observation.
\begin{lemma}\label{lem:OrdIJ}
With $I$, $J$ and $X_1,\ldots,X_n$ as in Lemma \ref{lem:Coupling}, $X_J\leq_\text{st}X_I$.
\end{lemma}
\begin{proof}
It suffices to show that $\mathbb{P}(X_I=1)\geq\mathbb{P}(X_J=1)$. By conditioning on the choices of $I$ and $J$ we have that
\[
\mathbb{P}(X_I=1)=\frac{\sum_{i=1}^np_i^2}{\sum_{j=1}^np_j}\text{ and }\mathbb{P}(X_J=1)=\frac{\sum_{i=1}^np_i(1-p_i)}{n-\sum_{j=1}^np_j}\,,
\]
and so the desired conclusion holds if and only if $n\sum_{i=1}^np_i^2\geq(\sum_{i=1}^np_i)^2$, which holds by the Cauchy--Schwarz inequality.
\end{proof}

We conclude the proof of Theorem \ref{thm:NegRel} by noting that for an increasing function $g:\{0,1,\ldots,n-1\}\to\mathbb{R}$ we have 
\begin{align}\label{eq:NegDepAppl}
\mathbb{E}[g(W-X_I)|X_I=1]&\leq\mathbb{E}g(W-X_I)
\leq\mathbb{E}g(W-X_J)\leq\mathbb{E}[g(W-X_J)|X_J=0]\,,
\end{align}
where the first inequality uses the negative relation property \eqref{eq:NegRel}, the second uses Lemma \ref{lem:OrdIJ} and that we may therefore couple our random variables in such a way that $W-X_I\leq W-X_J$ almost surely, and the final inequality again uses negative relation by writing
\begin{align*}
\mathbb{E}g(W-X_J)&=\mathbb{P}(X_J=1)\mathbb{E}[g(W-X_J)|X_J=1]+\mathbb{P}(X_J=0)\mathbb{E}[g(W-X_J)|X_J=0]\\
&\leq\mathbb{P}(X_J=1)\mathbb{E}g(W-X_J)+\mathbb{P}(X_J=0)\mathbb{E}[g(W-X_J)|X_J=0]\,,
\end{align*}
so that $\mathbb{E}g(W-X_J)\leq\mathbb{E}[g(W-X_J)|X_J=0]$. This completes the proof of Theorem \ref{thm:NegRel}.

\section{Concluding remarks and extensions} \label{sec:Conc}

We have shown that the binomial distribution with fixed support and mean maximises entropy within the class of sums of negatively related indicator random variables with that support and mean. This improves upon the main result of \cite{y08} by expanding the class of distributions within which the binomial has maximum entropy, and gives a sharp upper bound, unlike the result of \cite{d16} which compares sums of negatively related indicators to a Poisson distribution with matching mean.

In this concluding section we remark upon some related results and generalisations.

\subsection{Sums of positively related indicators}

It is natural to consider what can be said about sums of positively related indicator random variables. That is, a random variable $W=X_1+\cdots+X_n$ where the $X_i$ satisfy \eqref{eq:NegRel} with the inequality reversed for each $i$ and all increasing functions $g$. 
In this setting, following our proof of Theorem \ref{thm:NegRel} yields a convex ordering between $W$ and a binomial random variable in the special case of identically distributed indicator random variables, as in Theorem \ref{thm:PosRel} below.

Sums of identically distributed and positively related indicator random variables arise naturally in many areas, including matching problems (see Section 4.2 of \cite{bhj92}), subgraph counts and vertex degrees in random graphs (see, for example, Sections 5.1 and 5.2 of \cite{bhj92}), and extreme value statistics (see Section 8.3 of \cite{bhj92}).
\begin{theorem}\label{thm:PosRel}
Let $X_1,\ldots,X_n$ be positively related indicator random variables satisfying $\mathbb{E}X_i=p$ for all $i=1,\ldots,n$. Let $W=X_1+\cdots+X_n$ and $Z\sim\text{Bin}(n,p)$. Then $Z\leq_\text{cx}W$.  
\end{theorem}
\begin{proof}
With notation as in Lemma \ref{lem:Coupling}, we recall the observation in the proof of Theorem \ref{thm:NegRel} that we may choose $I=J$ almost surely in the identically distributed case. We may then write the following analogue of \eqref{eq:NegDepAppl}: For an increasing function $g:\{0,1,\ldots,n-1\}\to\mathbb{R}$ we have
\[
\mathbb{E}[g(W-X_I)|X_I=1]\geq\mathbb{E}g(W-X_I)\geq\mathbb{E}[g(W-X_I)|X_I=0]\,,
\]
where the two inequalities use the positive relation assumption analogously to the way we did in \eqref{eq:NegDepAppl} for negative relation. Hence, using the coupling from Lemma \ref{lem:Coupling} we may conclude that $W^+\leq_\text{st}W^\star$ in this case. Defining $W_t$ as in the proof of Theorem \ref{thm:StochOrder}, the arguments of Lemmas \ref{lem:Pres1} and \ref{lem:Pres2} show that this stochastic ordering is preserved by our Markov chain: $W_t^+\leq_\text{st}W_t^\star$ for all $t=0,1,\ldots$. With this stochastic ordering, the proof of Lemma \ref{lem:Convex} (reversing inequality signs where necessary) shows that $Z\leq_\text{cx}W$, as required. 
\end{proof}
While we have a convex ordering with a binomial distribution in Theorem \ref{thm:PosRel}, we cannot use the argument of \eqref{eq:LogConcave} to convert this into an entropy bound without the upper-bounding random variable (in this case, $W$) having a log-concave mass function. In general, this will not be the case for a sum of positively related indicators. Further investigation of entropy bounds for sums of positively related indicator random variables is left as a topic for future research.

\subsection{Compound distributions}

Finally, we note that results of \cite{y09} allow us to extend our main results to compound random variables:
\begin{theorem}
Let $W$ and $Z$ be as in Theorem \ref{thm:NegRel}. Let $\nu_1,\nu_2,\ldots$ and $\xi_1,\xi_2,\ldots$ be two sequences of independent and identically distributed random variables, each independent of $W$ and $Z$, such that $\nu_1\leq_\text{cx}\xi_1$. Let
\[
W^\prime=\sum_{i=1}^W\nu_i\text{ and }Z^\prime=\sum_{i=1}^Z\xi_i\,.
\]
If $Z^\prime$ has log-concave mass function, then $H(W^\prime)\leq H(Z^\prime)$.
\end{theorem}
\begin{proof}
This follows by combining the convex ordering $W\leq_\text{cx}Z$ we have shown in this setting with Theorem 1 of \cite{y09}, which was established by combining the argument of \eqref{eq:LogConcave} with closure properties of the convex order (see also Theorem 3.A.13 of \cite{ss07}).
\end{proof}

\end{document}